\documentclass[draft]{amsart} 
\usepackage{amssymb,latexsym,amsfonts,verbatim,amscd,ifthen} 
\usepackage{color}
\usepackage{mathrsfs}
\usepackage{url}
\usepackage{tikz}

\numberwithin{equation}{section}

\theoremstyle{plain}

\newtheorem{theorem}[equation]{Theorem}   %[section] 
\newtheorem*{Main Theorem}{Main Theorem}
\newtheorem{lemma}[equation]{Lemma} 
\newtheorem{proposition}[equation]{Proposition}

\theoremstyle{definition}

\newtheorem{definition}[equation]{Definition} 
\newtheorem{remark}[equation]{Remark} 
 
\newtheorem{example}[equation]{Example}

\begin{document}   

\renewcommand{\:}{\! :} 
\newcommand{\p}{\mathfrak p} 
\newcommand{\m}{\mathfrak m}
\newcommand{\e}{\epsilon}
\newcommand{\lra}{\longrightarrow} 
\newcommand{\lla}{\longleftarrow}
\newcommand{\ra}{\rightarrow} 
\newcommand{\altref}[1]{{\upshape(\ref{#1})}} 
\newcommand{\bfa}{\boldsymbol{\alpha}} 
\newcommand{\bfb}{\boldsymbol{\beta}} 
\newcommand{\bfg}{\boldsymbol{\gamma}} 
\newcommand{\bfM}{\mathbf M} 
\newcommand{\bfI}{\mathbf I} 
\newcommand{\bfC}{\mathbf C} 
\newcommand{\bfB}{\mathbf B} 
\newcommand{\bsfC}{\bold{\mathsf C}} 
\newcommand{\bsfT}{\bold{\mathsf T}}
\newcommand{\smsm}{\smallsetminus} 
\newcommand{\ol}{\overline} 
\newcommand{\os}{\overset}
\newcommand{\us}{\underset}

\newlength{\wdtha}
\newlength{\wdthb}
\newlength{\wdthc}
\newlength{\wdthd}
%\newcommand{\elabel}{\label}
%\newcommand{\mlabel}{\label}
%\reversemarginpar 
\newcommand{\elabel}[1]
           {\label{#1}  
            \setlength{\wdtha}{.4\marginparwidth}
            \settowidth{\wdthb}{\tt\small{#1}} 
            \addtolength{\wdthb}{\wdtha}
            \raisebox{\baselineskip}
            {\color{red} 
             \hspace*{-\wdthb}\tt\small{#1}\hspace{\wdtha}}}  

\newcommand{\mlabel}[1] 
           {\label{#1} 
            \setlength{\wdtha}{\textwidth}
            \setlength{\wdthb}{\wdtha} 
            \addtolength{\wdthb}{\marginparsep} 
            \addtolength{\wdthb}{\marginparwidth}
            \setlength{\wdthc}{\marginparwidth}
            \setlength{\wdthd}{\marginparsep}
            \addtolength{\wdtha}{2\wdthc}
            \addtolength{\wdtha}{2\marginparsep} 
            \setlength{\marginparwidth}{\wdtha}
            \setlength{\marginparsep}{-\wdthb} 
            \setlength{\wdtha}{\wdthc} 
            \addtolength{\wdtha}{1.1ex} 
            \marginpar{\vspace*{-0.3\baselineskip}%\color{blue}
                       \tt\small{#1}\\[-0.4\baselineskip]\rule{\wdtha}{.5pt} }
            \setlength{\marginparwidth}{\wdthc} 
            \setlength{\marginparsep}{\wdthd}  }

%\begin{document}

\title{Betti-linear Ideals}
\author[D. Wood]{Daniel Wood} 
\address{Department of Mathematics\\
         University at Albany, SUNY\\ 
         Albany, NY 12222}
\email{dwood@math.albany.edu}
%\date{\today} 
%\begin{abstract} 
%\end{abstract} 

\begin{abstract}
We introduce the notion of a Betti-linear monomial ideal, which generalizes the 
notion of lattice-linear monomial ideal introduced by Clark. We provide a characterization 
of Betti-linearity in terms of Tchernev's poset construction. As an 
application we obtain an explicit canonical construction for the minimal 
free resolutions of monomial ideals having pure resolutions.
\end{abstract}

\maketitle

\section*{Introduction}

\begin{comment} 
Let $k$ be a field. Consider the polynomial ring,
$R = k[x_1,\ldots,x_n]$ with the standard $\mathbb{Z}^n$-grading. Let $I \subseteq R$ be a 
monomial ideal with $\left\{m_1,\ldots, m_k\right\}$ a minimal monomial generating set for $I$ and 
we can consider the minimal free resolution of $I$ over $R$,
$$
\mathscr{F}: 0 \lla F_0 \os{\phi_1}{\lla} F_1 \lla \ldots
\lla F_{i-1} \os{\phi_i}{\lla} F_i \lla 
\ldots \lla F_n \lla 0
$$
where each $F_i$ is a free $\mathbb{Z}^n$-graded $R$-module and the differential 
$\phi_.$ preserves degrees. 
\end{comment} 

Understanding the structure of minimal free resolutions of monomial ideals 
is an active area of research in commutative algebra. One important aspect 
of this problem is investigating what it means for the resolution to be 
linear or close to linear.  
%Linearity of the resolution $\mathscr{F}$ is a desirable notion. 
For example, Eagon and Reiner~\cite{ER} have shown that an ideal $I$ with linear resolution has 
an Alexander dual that is Cohen-Macaulay; Herzog and Hibi~\cite{HH}  
introduce and study the notion of componentwise linear ideals; and 
Clark~\cite{Cl} 
introduced the notion of lattice-linearity and proved a criterion 
for $I$ to be 
lattice-linear in terms of acyclity properties of the so-called 
\emph{poset construction}. 

In this paper we introduce the new notion of a 
\emph{Betti-linear} monomial ideal %$\mathscr{F}=(F_k,\phi_k)$ 
that generalizes the notion of lattice-linearity. 
We provide a criterion for 
Betti-linearity in terms of the poset construction. A class of Betti-linear 
monomial ideals are those having pure resolution. 
They arise  
in connection with Boij-Soderberg theory \cite{BS,BS2,ES,EFW},  
%Monomial ideals with pure resolution 
and were studied recently by 
Francisco, Mermin, and Schweig~\cite{FMS}. Thus, our main result yields an explicit description of the structure of the minimal free resolution of $I$ when $I$ has pure resolution.

To be specific, 
let $B$ be the Betti poset \cite{CM,TV} of the monomial ideal $I$ with minimal 
free reolution $\mathcal F=(F_k,\phi_k)$. Thus, $B$ is 
the %sub-poset of $\mathbb{Z}^n$ 
set of monomial degrees of the basis elements of $\mathcal{F}$ ordered by divisibility.
%containing only those elements $\alpha$ for which one of the $F_i$ 
%contains a basis element of degree $\alpha$. 
We say that the ideal $I$ is \textit{Betti-linear} 
if we can fix homogeneous bases $B_k$ of the 
free modules $F_k$ for all k so that for any $i \geq 1$ and 
any $\tau \in B_i$, the element
$$
\phi_{i}(\tau)=\sum_{\gamma \in B_{i-1}}{\left[\tau:\gamma\right]\cdot \gamma}
$$
has the property that if $\left[\tau,\gamma\right] \neq 0$ then 
the monomial degree of $\gamma$ is covered in the poset $B$ by that of $\tau$.
In our main result, Theorem~\ref{main theorem}, we show that
$I$ is Betti-linear if and only if $\mathcal{F}$ can be recovered from the poset 
construction applied to the Betti poset of the ideal $I$.

%To prove this fact, we will examine the Betti poset 
%$B \subseteq L$. We will formalize an idea of 
%"Betti linearity" completely analogous to that of lattice-linearity.

%In order to prove this fact, we will need to discuss the poset construction
%due to Tchernev. 

The structure of this paper is as follows: In Section 2, we introduce the notion of Betti linearity,  examine a few examples, and state our main theorem. In Section 3 we discuss the poset construction in detail. Section 4 provides a few key properties of the Betti poset. Finally, in Section 5 we give a proof of our main theorem.

I would like to thank Alexandre Tchernev and Timothy Clark for their helpful discussions and insights.

\section{Preliminaries} 
Throughout this paper, $k$ is a fixed field and $R=k[x_1,\ldots,x_n]$ is a polynomial ring over $k$. 
The ring $R$ as a vector space over $k$ is the direct sum
\[
R= \bigoplus_{\alpha \in \mathbb{Z}^n} R_{\alpha}
\]
where $R_{\alpha}= k \cdot x^{\alpha}$ for $\alpha \in \mathbb{N}^n$ and is $0$ otherwise. 
%Observe that $R_{\alpha} \cdot R_{\beta} \subseteq R_{\alpha + \beta}$ with equality when both 
%$\alpha$ and $\beta$ are in $\mathbb{N}^n$. 
Thus, $R$ is a $\mathbb{Z}^n$-graded or \emph{multigraded} 
$k$-algebra. Let $\m = (x_1, \ldots, x_n)$ be the unique graded maximal ideal of the ring 
$R$. For $a = (a_1, \ldots, a_n) \in \mathbb{N}^n$, we write 
\[
x^{a} = x_{1}^{a_1} \cdots x_{n}^{a_n} \in R.
\]
%Give $R$ the standard $\mathbb{Z}^{n}$-grading (or multigrading). Using the multidegree map
%$mdeg : R \lra \mathbb{Z}^n$, defined by  
%\[
%mdeg(x^{a}) = a = (a_1, \ldots,a_n)
%\]
and we will always identify $x^a$ with its exponent $a \in \mathbb{Z}^{n}$.
For $\gamma \in \mathbb{Z}^n$ we write $R(-\gamma)$ for the shifted multigraded $R$-module %with 
%\[
%R(-\gamma)= \bigoplus_{\alpha \in \mathbb{N}^n} R(-\gamma)_{\alpha}
%\]
with $R(-\gamma)_{\alpha} = R_{\alpha - \gamma}$. Thus, $R(-\gamma)$ is free of rank 1 with basis a single homogeneous element of degree $\gamma$.

Let $\left(P,\leq\right)$ be a poset. Let $\sigma \subseteq P$ 
be a subset of P. If the meet or join of $\sigma$ exist, they are denoted as 
$\wedge\sigma$ and $\vee\sigma$ respectively. If $\sigma$ has the form 
$$
x_0 < x_1 < \ldots < x_k
$$
then $\sigma$ is called a \textit{chain of length k} or a \textit{k-chain} of $P$. For any element $x \in P$,
we define the \textit{dimension} of $x$ to be 
$$
d_P(x)=d(x)=sup\left\{k:x_0<\ldots<x_k=x\right\}
$$

Any subset of $P$ that is comprised of elements that are pairwise 
incomparible is called an \textit{anti-chain}. An element $y \in P$ is said to 
be \textit{covered} by $x$, which we denote $y \lessdot x$, when it is true that 
$y < x$ and there exists no $z \in P$ so that $y<z<x$. Denote by $P_{<x}$ the subset of $P$ given by
$$
P_{<x}=\left\{z \in P : z<x\right\},
$$
with $P_{\leq x}$ defined analogously. Let $A$ be the set of minimal elements of $P$,
and note that $a \in A$ if and only if $d(a)=0$ %We say the poset $P$ is \textit{ranked} if it
%is true that $d(x)=d(y)+1$ for all $y \lessdot x \in P$.

To a poset $P$, we associate its order complex $\Delta\left(P\right)$ which is an abstract simplicial  
complex whose vertices are the elements of $P$ and for each $k>0$, the 
$k$-dimensional faces of $\Delta\left(P\right)$ are  
the $k$-chains of $P$. When we refer to the topological properties 
of the poset $P$, we are referring to the topological properties of the abstract simplicial 
complex $\Delta\left(P\right)$.

Conversely, given a simplicial complex $S$, one may define the face poset of $S$, $F\left(S\right)$, 
which is the set of nonempty faces of $S$ partially ordered by inclusion. Under these correspondences, we can 
identify the first barycentric subdivision of $S$ as \textbf{sd}$\left(S\right)= \Delta\left(F\left(S\right)\right)$.

\section{Betti-linearity}
%Let $k$ be a field. Consider the polynomial ring
%$R = k[x_1,\ldots,x_n]$. 
Let $I \subseteq R$ be a monomial ideal  
%Let $L$ be the lcm-lattice 
%of the multidegrees of the minimal generators of the ideal $I$. 
%Let $B \subseteq L$ be 
%the sub-poset of the lcm-lattice defined by $\alpha \in B \iff$ 
%there is $i \in \left\{1,\ldots, n\right\}$
%so that $\beta_{i,\alpha} \ne 0$. Denote by $d_{B}(\alpha)$ the dimension of 
%the element $\alpha$ in the poset $B$.
with minimal $\mathbb{Z}^n$-graded free resolution
\[
\mathcal{F}: 0 \lla F_0 \os{\phi_1}{\lla} F_1 \os{\phi_2}{\lla} \ldots
\os{\phi_{i-1}}{\lla} F_{i-1} \os{\phi_i}{\lla} F_i \os{\phi_{i+1}}{\lla} 
\ldots \os{\phi_m}{\lla} F_m \lla 0.
\]
For each $i$, write 
\[
mdeg:B_i \lra \mathbb{Z}^n
\] for the map that assigns to a basis element 
$\tau \in B_i$ its $\mathbb{Z}^n$-degree. In particular, $mdeg(\tau)$ is an element of the Betti 
poset $B$ of $I$ over $k$. We also write $deg(\tau)$ for the total degree of $\tau$ in $\mathbb{Z}$.

%\begin{definition}
%Consider the collection 
%\[
%BLS_{t} = \left\{ \bigoplus_{\alpha \in \mathscr{B}}{Re_\alpha} : ht_{\mathscr{B}}(\alpha) = t \right\}
%\]
%where $e_\alpha$ is a basis element of the free module $F_t$. We will call the collection $BLS_t$ the 
%\emph{Betti-Linear} strand of the minimal free resolution of $I$ in homological dimension $t$. 
%\end{definition}

\begin{definition}
The ideal $I$ is called \emph{Betti-linear} if in the minimal free resolution $\mathcal{F}$ we can fix a $\mathbb{Z}^n$-graded basis $B_t$ of $F_t$ for each $t$ so that for all $i \geq 1,$ and 
for all $\tau \in B_i$, we have that 
\[
\phi_{i}(\tau) = \sum_{\gamma \in B_{i-1}}{[\tau:\gamma]\gamma}
\]
has that property that if $[\tau:\gamma] \neq 0$ then $mdeg(\gamma) \lessdot_{B} mdeg(\tau)$.
\end{definition}

An important invariant of $I$ is its \emph{lcm-lattice} consisting of the joins in 
$\mathbb{N}^n$ of subsets of the degrees of the minimal generators of $I$. 
We denote by $L$ the lcm-lattice of $I$ without its minimal element. It is well known that the Betti 
poset is a subposet of $L$, thus the notion of Betti-linearity generalizes the previously established notion of 
lattice-linearity due to Clark \cite{Cl}. %We recall the definition here:

%\begin{definition} \cite{Cl} The ideal $I$ is called \emph{lattice-linear} if in the minimal 			%free resolution of $I$ we can fix for each $t$ the $\mathbb{Z}^m$-graded bases $B_t$ of $F_t$ so that %for all $i \geq 1$ and every $\tau \in B_i$ we have that 
%			\[
%			\partial_i(\tau)=\sum_{\gamma \in B_{i-1}}[\tau:\gamma]\gamma
%			\]
%has the property that $[\tau:\gamma] \neq 0$ implies that $\tau$ is a cover of $\gamma$ 
%in $L$ 
%\end{definition}

\begin{example}
Let $k$ be any field. Let $R=k[a,b,c,d,e]$ be the polynomial ring over $k$ in 5 variables. 
Consider the ideal $I=(ac,bd,ae,de) \subseteq R$. The minimal free resolution of this ideal takes the form 
				\[
				0 \longleftarrow R \longleftarrow R^4 \longleftarrow R^4 \longleftarrow R \longleftarrow 0
				\]
and is known to be Betti-linear, however, the ideal $I$ does not have a lattice-linear resolution. For the given ideal $I$, the lcm-lattice of $I$ is 
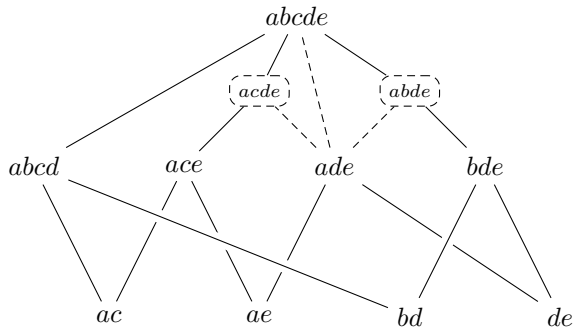
\[
				\begin{tikzpicture}
[notinB/.style={rectangle, rounded corners=1.5mm, 
                densely dashed, draw=black}]
\node (abcde)          at (-.5,4) {$abcde$};
\node (acde)  [notinB] at  (-1,3) {$\scriptstyle acde$};
\node (abde)  [notinB] at  (+1,3) {$\scriptstyle abde$};
\node (abcd)           at  (-4,2) {$abcd$};
\node (ace)            at  (-2,2) {$ace$};
\node (ade)            at   (0,2) {$ade$};
\node (bde)            at  (+2,2) {$bde$};
\node (ac)             at  (-3,0) {$ac$};
\node (ae)             at  (-1,0) {$ae$};
\node (bd)             at  (+1,0) {$bd$};
\node (de)             at  (+3,0) {$de$};
\draw (de) to (ade) to (ae) to (ace) to (ac) to (abcd) 
           to (abcde) to (abde) to (bde) to (de)
      (abcde) to (acde) to (ace);
\draw[preaction={draw=white, -, line width=5pt}] 
      (bde) to (bd) to (abcd);
\draw[densely dashed] (ade) to (acde)
                      (ade) to (abde)
                      (ade) to (abcde);
\end{tikzpicture}
		\]
Here, the dashed entries indicate elements $\alpha$ for which $\alpha \notin B$.
\end{example}

\begin{example}
Let $k$ be any field. Let $R=k[a,b,c,d]$ be a polynomial ring over $k$. 
Consider 
the ideal 
\[
I = (ad, bc^2, c^3d^4, a^2, ab, ac, b^2c, c^4d^3, b^3, c^5d^2, cd^6, c^7).
\]
This is a stable ideal of $R$ with respect to the ordering $a<b<c<d$. 
The formulas of Eliahou and Kervaire~\cite{EK} show that in 
homological degree $1$ of the minimal resolution of $I$ there is a single  
direct summand $R\bigl(-(1,1,2,0)\bigr)$, and provide a generator $\tau$ 
for it. They also show that in the same homological degree 
there is a unique (up to a constant multiple) 
basis element $\sigma$ with 
%$\tau$ satisfying 
\[
mdeg(\sigma) = (1,1,1,0) ; 
%mdeg(\tau) = (1,1,2,0).
\]
and that all other basis elements degrees are not comparable with the 
degree of $\tau$. 
Furthermore, $\tau$ maps onto an element with a component that is a 
non-zero multiple of the unique basis element  
$\eta$ in homological degree $0$ with $mdeg(\eta) = (1,1,0,0)$. 
Clearly, one has that $\eta < \sigma < \tau$ in $B$, so that $\tau$ is not 
a cover for $\eta$. Since any other choice of a homogeneous basis element 
of degree $(1,1,2,0)$ in homological degree $1$ has to be of the form 
$r\tau + sc\sigma$ for some constants $r,s\in k$ with $r\ne 0$, 
it is immediate that any such choice will violate  
the Betti-linearity condition as well. 
It follows that the stable ideal $I$ is not Betti-linear. 
\end{example}

Next, we show that monomial ideals with pure resolution are in fact  
Betti-linear ideals:

\begin{proposition}
Let $I \subseteq R=k[x_1, \ldots, x_n]$ be a monomial ideal with pure resolution. Then $I$ is a 
Betti-linear ideal.
\end{proposition}

\begin{proof}
Suppose that the minimal free resolution of $I$ is 
\[
\mathcal{F}: 0 \lla F_0 \os{\partial_1}{\lla} F_1 \os{\partial_2}{\lla} \ldots
\os{\partial_{i-1}}{\lla} F_{i-1} \os{\partial_i}{\lla} F_i \os{\partial_{i+1}}{\lla} 
\ldots \os{\partial_k}{\lla} F_k \lla 0
\]
then because this resolution is pure, it must be the case that for each basis element 
$\sigma \in F_i$ that $deg(\sigma) = d_i$. Suppose now that $\sigma \in F_i$ and $\tau \in F_{i-1}$ are 
basis elements such that $[\sigma : \tau] \neq 0$. Clearly, $\tau < \sigma$. We want to show that 
$\tau \lessdot \sigma$.

Suppose that there was $\gamma \in B$ such that $\tau < \gamma < \sigma$. The resolution is 
minimal, so therefore $deg(\tau)=d_{i-1} < deg(\gamma) < d_i = deg(\sigma)$. This is impossible, as the elements of 
$B$ must have total degree which is an element of the set $\left\{ d_0, d_1, \ldots, d_{k-1}, d_k \right\}$, a strictly increasing sequence.

It follows that no such $\gamma$ exists and that $\tau \lessdot \sigma \in B$.
\end{proof}

%\begin{theorem} \cite{Cl}
%				%Suppose that $R=k[x_1, \ldots, x_m]$ is the multigraded 
%				%polynomial ring over the field $k$ in $m$ variables. 
%				A monomial ideal 
%				$I \subseteq R$ is lattice-linear if and only if the 
%				poset construction on $L$ recovers the minimal free resolution of the ideal $I$. 
%\end{theorem}

The main theorem of this paper provides a characterization for Betti-linear ideals that generalizes  \cite[Theorem 3.3]{Cl}:

\begin{theorem}\label{main theorem}
				%Suppose that $R=k[x_1, \ldots, x_m]$ is the multigraded polynomial ring over the 
				%field $k$ in $m$ variables. 
				A monomial ideal $I \subseteq R$ is 
				Betti-linear if and only if the poset construction on 
				$B$ recovers the minimal free resolution of the ideal $I$. 
				In particular, this provides an explicit canonical construction of the minimal free resolution 
				of monomial ideals with pure resolution.
\end{theorem}

We postpone the proof of Theorem~\ref{main theorem} until Section 5.

\section{The Poset Construction}

Let $P$ be a poset. For $\alpha \in P$, let $\Delta_{\alpha} = \Delta(P_{<\alpha})$ and 
$\Delta_{\leq \alpha} = \Delta(P_{\leq \alpha})$. 
Thus
$$
\Delta_{\alpha} = \bigcup_{\lambda \lessdot \alpha} \Delta_{\leq \lambda}.
$$
Fix $\lambda \lessdot \alpha$, and set
\[
\Delta_{\alpha,\lambda}= \Delta_{\leq \lambda}\cap (\bigcup_{\lambda \neq \beta \lessdot \alpha} \Delta_{\leq \beta})
\].

\begin{definition}\label{D:one} \cite{Cl}
For $i \geq 0$, 
set $\mathscr{D}_{i,\alpha}=\widetilde{H}_{i-1}(\Delta_{\alpha}, k)$ and set
\[
\mathscr{D}_{i} = \bigoplus_{\alpha \in P} \mathscr{D}_{i, \alpha}
\]
\end{definition}

\begin{remark}
If $i=0$ and $\alpha \in A$, then $\Delta_{\alpha}$ is the empty simplicial complex, 
and so we see that 
\[
\mathscr{D}_{0,\alpha} = \widetilde{H}_{-1}(\{\emptyset\}, k)
\]
a $1$-dimensional $k$-vector space. On the other hand, if $i=1$ and $\alpha \notin A$, then 
$\Delta_{\alpha} \neq \{\emptyset\}$. From this it 
follows that $\mathscr{D}_{0, \alpha} = \widetilde{H}_{-1}(\Delta_{\alpha}, k) = 0$. Thus, 
\[
\mathscr{D}_{0} = \bigoplus_{\alpha \in A} \mathscr{D}_{0, \alpha} = 
\bigoplus_{\alpha \in A}\widetilde{H}_{-1}(\{\emptyset\}, k) \cong \bigoplus_{\alpha \in A}k.
\]
\end{remark}

Given $\lambda \lessdot \alpha \in P$, consider the Mayer-Vietoris sequence in reduced homology for the triple 
\[
   ( \Delta_{\leq \lambda}, \bigcup_{\lambda \neq \beta \lessdot \alpha} \Delta_{\leq \beta}, \Delta_{\alpha} )
\]
Write $j:\widetilde{H}_{i-2}(\Delta_{\alpha,\lambda},k) \lra 
\widetilde{H}_{i-2}(\Delta_{\lambda},k)$ for the map induced in homology by the inclusion map and let
\[
\partial^{\alpha,\lambda}_{i-1}:\widetilde{H}_{i-1}(\Delta_{\alpha},k) \lra 
\widetilde{H}_{i-2}(\Delta_{\alpha,\lambda},k)
\]
be the connecting homomorphism from the Mayer-Vietoris sequence. Recall that for $[c] \in \widetilde{H}_{i-2}(\Delta_{\alpha}, k)$, this homomorphism 
is given by 
\[
\partial^{\alpha,\lambda}_{i-1}([c]) = [d_{i-1}(c')] \in \widetilde{H}_{i-2}(\Delta_{\alpha, \lambda}, k)
\]
where we have that $c' + c'' = c \in \widetilde{C}_{i-2}(\Delta_{\alpha}, k)$ and $c'$ and $c''$ are 
any components of $c$ that are supported by $\Delta_{\leq \lambda}$ and $\bigcup_{\lambda \neq \beta \lessdot \alpha} \Delta_{\leq \beta}$ 
respectively, and $d$ is the usual simplicial boundary map.

\begin{definition} \cite{Cl}
For $i \geq 1$, define $\phi_{i}: \mathscr{D}_i \lra \mathscr{D}_{i-1}$ componentwise by 
\[ 
\phi_{i}|_{\mathscr{D}_{i,\alpha}} = \sum_{\lambda \lessdot \alpha} \phi^{\alpha,\lambda}_{i}
\]
where the map 
\[
\phi^{\alpha,\lambda}_{i}: \mathscr{D}_{i,\alpha} \lra \mathscr{D}_{i-1,\lambda}
\]
is the composition $\phi^{\alpha,\lambda}_{i} = j \circ \partial^{\alpha,\lambda}_{i-1}$.
We define $\mathscr{D}(P,k)$ as the sequence of modules and maps 
\[
\mathscr{D}(P,k): \mathscr{D}_0 \os{\phi_1}{\lla} \mathscr{D}_1 \lla \ldots 
\mathscr{D}_{i-1} \os{\phi_i}{\lla} \mathscr{D}_i \ldots \lla \mathscr{D}_k \lla 0,
\]
and we refer to $\mathscr{D}(P,k)$ as the \emph{poset construction} on $P$ over $k$.
\end{definition}
Notice that the vector space maps $\phi_i$ are canonical and determined by the structure of the 
homology of the filters $P_{\leq \alpha}$ in the poset $P$.

%\begin{remark}
%The sequence $\mathscr{D}(P)$ is not necessarily a complex of vector spaces a priori. Further, even if 
%it is a complex, it need not be exact. While the necessary and sufficient conditions to produce an %exact complex are unknown, 
%it is known that if the poset $P$ is ranked, then $\mathscr{D}(P)$ is a complex of vector spaces.
%\end{remark}

Suppose that $\eta : P \lra \mathbb{Z}^{n}$ is a map of partially ordered sets and $A$ is the set of minimal elements of the poset $P$. Let $N \subseteq R$ be the ideal 
minimally generated by the set
\[
\{x^{\eta(a)} : a \in A\}.
\]
Then the sequence $\mathscr{D}(P,k)$ can be 
homogenized using the map $\eta$ to produce a sequence of multigraded $R$-modules and $R$-module 
homomorphisms which will approximate a free resolution of the multigraded module $R/N$.

We homogenize in the following way: For $i \geq 0$, set 
\[
\mathscr{F}_{i}(\eta) = \bigoplus_{\lambda \in P} \mathscr{F}_{i, \lambda}(\eta) = 
\bigoplus_{\lambda \in P} \mathscr{D}_{i, \lambda} \otimes_{k} R(-\eta(\lambda))
\]
and note the multigrading satisfies $mdeg(v \otimes x^{a}) = a + \eta(\lambda)$ for each $v \in \mathscr{D}_{i, \lambda}$. 

The differential in this sequence is defined componentwise in homological degree  $i \geq 1$ by 
the map $\partial_{i} : \mathscr{F}_{i}(\eta) \lra \mathscr{F}_{i-1}(\eta)$ given by 
\[
\partial_{i}|_{\mathscr{F}_{i, \alpha}(\eta)} = \sum_{\lambda \lessdot \alpha} \partial^{\alpha, \lambda}_{i},
\]
where $\partial^{\alpha, \lambda}_{i} : \mathscr{F}_{i, \alpha}(\eta) \lra \mathscr{F}_{i-1, \lambda}(\eta)$ has the 
form $\partial^{\alpha, \lambda}_{i} = x^{\eta(\alpha) - \eta(\lambda)} \otimes \phi^{\alpha, \lambda}_{i}$ for $\lambda \lessdot \alpha$. 
This gives a sequence of multigraded $R$-modules and maps
\[
\mathscr{F}(\eta) : \ldots \lra \mathscr{F}_{i}(\eta) \overset{\partial_i}{\lra} \mathscr{F}_{i-1}(\eta) \lra
\ldots \lra \mathscr{F}_{1}(\eta) \overset{\partial_1}{\lra} \mathscr{F}_{0}(\eta).
\]

\begin{definition} \cite{Cl}
If $\mathscr{F}(\eta)$ is an acyclic complex of multigraded modules, it is 
called a $\textit{poset resolution}$ of the ideal $N$.
\end{definition}

\section{Properties of Betti posets}

Let $I$ be a monomial ideal of $R$. 
%Recall that $B$ denotes the Betti poset of $I$, the sub-poset of the lcm-lattice of $I$ containing only those elements $\alpha$ for which one of the $F_i$ contains $R(-\alpha)$ as a summand. 
Recall that $L$ denotes the lcm-lattice of $I$ without its minimal element, and 
denote the inclusion of the Betti poset $B$ into $L$ by $\iota: B \lra L$. The formulas of Gasharov, Peeva, and Welker \cite{GPW} tell us that $\alpha \in L-B$ if and only if 
$\widetilde{H}_*(\Delta(L_{<\alpha}),k)=0$. By 
Tchernev and Varisco \cite{TV} and Clark and Mapes \cite{CM}, the induced map 
on homology $\iota_*: H_{*}(\Delta(B),k) \lra H_{*}(\Delta(L),k)$ 
is an isomorphism.

%The first thing to understand is how the inverse of this isomorphism is realized. 
%Indeed, if there is a sequence of covers $\gamma \lessdot \beta \lessdot \alpha \in L$ so that 
%$\alpha, \beta \in B$ but $\gamma \notin B$, then observe that $d_{L}(\alpha) \neq d_{B}(\alpha)$. 

To analyze the map $\iota$ better, we consider the following simplicial complex associated to any poset $P$. Recall that $A$ is the set of minimal elements of $P$.

\begin{definition}
%$F=\left\{a_{i_1}, \ldots, a_{i_k}\right\} \subseteq A$ be a maximal with respect to the property that 
%$\bigvee_{j=1}^{k} \left\{a_{i_j}\right\}$ is bounded in $P$. 

(a) We define $\Theta(P)$ to be the abstract simplicial complex on the set of vertices $A$ with faces the collection of all $F \subseteq A$ so that $F$ is bounded in $P$.

(b) When $\rho:P \lra Q$ is a morphism of posets sending minimal elements of $P$ to minimal elements of $Q$, we write $\Theta(\rho)$ for the induced map of simplicial complexes $\Theta(P) \lra \Theta(Q)$.

\end{definition}

We briefly discuss the connection between $\Theta(P)$ and crosscut complexes. Recall that a subset $C \subseteq P$ of a poset $P$ is called a $\emph{crosscut}$ if $(1)$ $C$ is an antichain, 
$(2)$ for every finite chain $\sigma$ in $P$ there exists some element in $C$ which is comparable to every element in $\sigma$, and 
$(3)$ if $E \subseteq C$ is bounded, then the join $\vee E$ or the meet $\wedge E$ exist in $P$. Here, when we say $E$ is bounded, we mean that $E$ has an upper or a lower bound in the poset $P$. 

If $C \subseteq P$ is a crosscut, then the crosscut complex $\Gamma(P,C)$ is defined to be the simplicial complex consisting of the bouned subsets of C. It is well known that $\Gamma(P,C)$ is homotopy equivalent to $\Delta(P)$. 

Now observe that if $A$ forms a crosscut of $P$, then $\Gamma(P,A)=\Theta(P)$. In particular, since $A$ is a crosscut of 
$L$, we obtain that $\Gamma(L,A) = \Theta(L)$ and hence 
is homotopy equivalent to $\Delta(L)$.

%we will make use of two other complexes. Let $\Gamma := \Gamma(L)$ be the $\emph{cross-cut}$ complex %on the lattice $L$. Similarly, let $\Gamma(L_{<\alpha})$ be the cross-cut complex on subposet of $L$ %consisting of all elements strictly less than $\alpha$. 

%Since the Betti poset $B$ is not necessarily a lattice, we can not consider it's crosscut complex. %However, for each $\alpha \in B$ we can 
%consider a simplicial complex we will denote $\Theta(B_{<\alpha})$ which is defined by facets all %possible $F$ which are maximal subsets of the atoms of 
%$B$ less than $\alpha$ so that $lcm(F) \neq \alpha$. Then define
%\[
%\Theta := \Theta(B) = \bigcup_{\alpha \in B} \Theta(B_{<\alpha})
%\]

%It is well known that $\Gamma$ is homotopy equivalent to $\Delta(L)$. We will show further that 
%$\Theta$ is homologically isomorphic to both $\Gamma$ and $\Delta(B)$.

Next, consider the map $\Theta(\iota): \Theta(B) \lra \Theta(L)$. 
We would like to show that the map $\Theta(\iota)_* : \widetilde{H}_{i}(\Theta(B_{<\alpha}),k) \lra 
\widetilde{H}_{i}(\Theta(L_{<\alpha}),k)$ 
is an isomorphism for all $i$. To do so, we will make use of the following key lemma:

\begin{lemma}
The relative homology groups $\widetilde{H}_{i}(\Theta(L_{<\alpha}),\Theta(B_{<\alpha}),k)$ are $0$ for all $i$.
\end{lemma}

\begin{proof}
Let $c \in C_{i}(\Theta(L_{<\alpha}),k)$ be an $i$-chain 
such that $\partial_{i}c \in \Theta(B_{<\alpha})$. Then it 
is enough to show that there exists $b \in \Theta(L_{<\alpha})$ so that $c - \partial_{n}b \in \Theta(B_{<\alpha})$.

We know that $c$ is an element of $C_{i}(\Theta(L_{<\alpha}),k)$, 
so we may write 
\[
c = \sum_{F \in \Theta(L_{<\alpha})}a_{F}F,
\] 
and without loss of generality, we may assume that each $F$ is not in $\Theta(B_{<\alpha})$. 
For each $F$, consider $\alpha_F = \bigvee_{L}\left\{a : a \in F\right\}$, hence 
$\left\{\alpha_1, \ldots, \alpha_k\right\}$ are joins of the faces $F$ and $\alpha_j \notin B$. 
Let $\alpha_1, \ldots, \alpha_p$ be the maximal elements among all the $\alpha_j$. It is necessarily true that there is no $\beta \in B$ so 
that $\alpha_j < \beta < \alpha$ in $L$ for any $j$. 

For each $j$, let $z_{\alpha_j} = \sum_{\bigvee(F)=\alpha_j}a_{F}F$. Then of course we have that $c = z_{\alpha_1} + \ldots + z_{\alpha_k}$, hence, 
\[
\partial_{i}z_{\alpha_1} = \partial_{i}c - \sum_{2 \leq j \leq k}\partial_{i}z_{\alpha_j}
\]

For a chain $\sigma$ with $\partial\sigma = \sum c_{m}\tau_m$, define 
$supp(\partial\sigma) = \left\{\tau_m : c_m \neq 0\right\}$.
Now suppose that $F$ is a face of $supp(\partial_{i}z_{\alpha_1})$ such that $\alpha_F = \alpha_1$. Then notice that $F$ is not in the support of 
$\sum_{2 \leq j \leq k}\partial_{i}z_{\alpha_j}$, so therefore we must have that $F \in supp(\partial_{i}c)$. This however implies that $F \in \Theta(B_{<\alpha})$, which implies 
there is $\beta \in B$ so that $\alpha_1 < \beta < \alpha$. Since this is a contradiction, we must conclude that each face $F \in supp(\partial_{i}z_{\alpha_1})$ satisfies $\alpha_F < \alpha_1$. 

From this we see that $\partial_{i}z_{\alpha_1} \in \Theta(L_{<\alpha_n})$ while $z_{\alpha_1} \in 
\Theta(L_{\leq \alpha_1})$. Therefore, $z_{\alpha_1}$ is a relative cycle of the pair $(\Theta(L_{\leq \alpha_1}), \Theta(L_{<\alpha_1}))$. However, the 
relative chain complex of the pair $(\Theta(L_{\leq \alpha_1}), \Theta(L_{<\alpha_1}))$ is acyclic over $k$ as $\Theta(L_{\leq\alpha})$ is a simplex and $\Theta(L_{<\alpha_1})$ is acyclic over $k$ as it is 
homotopy equal to $\Gamma(L_{<\alpha_1})$ and
$\alpha_n \notin B$. 
So we get that $z_{\alpha_1}$ is also a relative boundary. Due to this, we can find $b_1 \in \Theta(L_{\leq \alpha})$ so that $z_{\alpha_1} - \partial_{i}b_{1} \in \Theta(L_{<\alpha})$. 

Repeat this for $\alpha_2, \dots, \alpha_p$ to get $b_2, \ldots, b_p$. Then consider the element 
\[
c_1 := c - \sum_{1 \leq j \leq p}\partial_{i}b_{j}.
\] 
The element $c_1$ may not yet satisy the conditions of the lemma, but for certain, if $G$ is a face appearing in the expression for $c_1$, we have that 
$\alpha_G < \alpha_j$ for (at least) one of the $\alpha_j, j=1,\ldots,p$. 

If $c_1$ does not  satisfy the lemma, we may iterate the same procedure. Letting $l = max\left\{d(\alpha_1), \ldots, d(\alpha_p)\right\}$, we see that after obtaining $c_l$, 
the process has terminated with $c_l \in \Theta(B_{<\alpha})$ as desired.
\end{proof}

\begin{proposition}
The map $\Theta(j)_* : \widetilde{H}_{i}(\Theta(B_{<\alpha}),k) \lra \widetilde{H}_{i}(\Theta(L_{<\alpha}),k)$ 
is an isomorphism for all $i$.
\end{proposition}

\begin{proof}
Consider the long exact sequence in homology
\begin{align*}
\ldots \lra \widetilde{H}_{i+1}(\Theta(L_{<\alpha}),\Theta(B_{<\alpha}),k) \lra 
\widetilde{H}_{i}(\Theta(B_{<\alpha}),k) \lra \widetilde{H}_{i}(\Theta(L_{<\alpha}),k) \\
               \widetilde{H}_{i}(\Theta(L_{<\alpha}),\Theta(B_{<\alpha}),k) \lra \ldots
\end{align*}
and apply the lemma above.
%, we see this simplifies to 
%\[
%\ldots \lra 0 \lra \widetilde{H}_{i}(\Theta(B_{<\alpha}),k) \lra %\widetilde{H}_{i}(\Theta(L_{<\alpha}),k) 
%\lra 0 \lra \ldots
%\] 
%and because the sequence is exact the result follows for all $i$.
\end{proof} 

Consider the map $\Psi: \Delta(L) \lra \Theta(L)$ given by 
\[
\sigma \mapsto \left\{x \in A : \sigma \in \Delta(L_{\geq x})\right\}
\]
where $L_{\geq x} = \left\{y \in L : y \geq x \right\}$. Note that $\Psi$ is precisely the map that realizes the homotopy equivalence between $\Delta(L)$ and $\Theta(L)$, see \cite{Bj}. Restricting $\Psi$ to $B$, we obtain the map $\psi: \Delta(B) \lra \Theta(B)$ given by 
\[
\sigma \mapsto \left\{ x \in A : \sigma \in \Delta(B_{\geq x}) \right\}.
\]

%Now we study the map $\psi: \Delta(B) \lra \Theta(B)$. To this end, we will study the corresponding map 
%$\Psi: \Delta(L) \lra \Theta(L)$. 

While $\Psi$ is always a homotopy equivalence, in general it will not be the case that $\psi$ is a homotopy equivalence. However we can show that it is an isomorphism in homology. For this we use the following lemma:

%\begin{proposition}
%Let $\Delta$ be a simplicial complex and $(\Delta_i)_{i \in I}$ a family of subcomplexes such that %$\Delta = \bigcup_{i \in I} \Delta_i$. 
%Then if every nonempty finite intersection $\Delta_{i_1} \cap \Delta_{i_2} \cap \ldots \cap %\Delta_{i_k}$ is homologically trivial over a fixed 
%field, $k$. Then $\Delta$ and the nerve $\mathscr{N}(\Delta_i)$ are homologically isomorphic.
%\end{proposition}
%
%\begin{proof}
%For each $J \subseteq I$, we have that $\widetilde{H}_{*}(\bigcap_{j \in J}\Delta_{j} ; k) = 0$. %Consider the map 
%\[
%g_j: \Delta \stackrel{sd}{\lra} sd(\Delta) \stackrel{f}{\lra} sd(\mathscr{N}(\Delta_j)) %\stackrel{unsd}{\lra} \mathscr{N}(\Delta_j)
%\]
%where the map $f: sd(\Delta) \lra sd(\mathscr{N}(\Delta_j))$ is the order reversing map so that for %$\sigma \in sd(\Delta)$ we have that 
%\[
%f(\sigma) = \left\{ j \in J : \sigma \in \Delta_j \right\}
%\]
%
%Let $P$ be the face poset of the complex $\Delta$. Let $Q$ be the face poset of the nerve considered %with reverse order. 
%Then for any $q \in Q$, we have that $\Gamma := f^{-1}(Q_{\leq q}) = \bigcap_{i=1}^{l}\Delta_{s_i}$. %However, based on the assumption that every nonempty finite intersection $\Delta_{i_1} \cap %\Delta_{i_2} \cap \ldots \cap \Delta_{i_k}$ is homologically trivial over the field $k$, we see that %$\Delta(\Gamma)$ is acyclic. By applying Cor. 4.3 in \ref{BWW}, we get that $\widetilde{H}_{r}(P) %\cong \widetilde{H}_{r}(Q)$ for all values of $r$, with the same result holding over any field. 
%\end{proof}

\begin{lemma}
The diagram
%\[
\begin{equation}
\begin{CD}
C(\Delta(B),k) @>{\Delta(j)}>> C(\Delta(L),k) \\
@V{sd}VV                   @VV{sd}V \\
C(sd(\Delta(B)),k) @>{sd(\Delta(j))}>> C(sd(\Delta(L)),k) \\
@V{sd(\psi)}VV                          @VV{sd(\Psi)}V \\
C(sd(\Theta(B)),k) @>{sd(\Theta(j))}>> C(sd(\Theta(L)),k) \\
@A{sd}AA                        @AA{sd}A \\
C(\Theta(B),k)   @>{\Theta(j)}>>  C(\Theta(L),k) \\
\end{CD}
\label{eq:A}
\end{equation}
%\]
is commutative. 
\end{lemma}

\begin{proof}
The top and bottom square commute for trivial reasons. 
%Indeed, if 
%$\left\{a_0 < \ldots < a_k\right\} = \sigma$ is a face of $\Delta(B)$, then the barycentric subdivision
%$sd(\sigma)$ has a face for every chain $\sigma_1 < \sigma_2 < \ldots < \sigma_n =\sigma$ in $B$. %Since the horizontal maps are inclusions, we see that the top 
%square in the diagram is commutative, as $sd(\sigma)$ is the same regardless 
%whether $\sigma$ is considered as a face of $\Delta(B)$ or of $\Delta(L)$. 

To show that the middle square is commutative, let $\Psi: \Delta(L) \lra \Theta(L)$ 
%be the map defined by $h(\sigma) = \left\{a \in A : \sigma \in \Delta(L_{\geq a})\right\}$. Of course, the map $h$ gives rise to the map 
%in the diagram, 
%\[
%f: \Delta(P(\Delta(L)) = sd(\Delta(L)) \lra \Delta(\Theta(L) = sd(\Theta(L)).
%\]
and note that $sd(\Psi)([\sigma_0 \subset \ldots \subset \sigma_k]) = [\Psi(\sigma_k) \subset \ldots \subset \Psi(\sigma_0)]$, or $0$ 
if the simplex defined by this formula is degenerate. Assume that each $\sigma_i \in 
P(\Delta(B))$. Then $\Psi(\sigma_i) = \left\{a \in A : \sigma \in \Delta(L_{\geq a})\right\} = \left\{a \in A : \sigma \in \Delta(B_{\geq a})\right\}$, the last equality 
being an equality of sets, as $\sigma_i \in P(\Delta(B))$ tells us that if $\sigma \in \Delta(L_{\geq a})$, 
then we also have $\sigma \in \Delta(B_{\geq a})$. We conclude that $sd(\psi)$ has 
image contained in $sd(\Theta(B))$. %Again, since the horizontal maps are inclusions, we can conclude that the middle square also commutes. 
%Finally, we show that the top square commutes. Let $A=\left\{a_1, \ldots, a_l\right\}$ be the set of %atoms of $L$. 
%Order $A$ by $a_i < a_j$ whenever $i<j$. Then the map $unsd: sd(\Theta(L)) \lra \Theta(L)$ is defined 
%by mapping a face to the smallest atom contained in it under our fixed order. Let $\sigma_0 \subseteq %\ldots 
%\subseteq \sigma_k$ be a chain of faces such that for each $i$, $\sigma_i \in \Theta(L)$. For each %$i$, suppose that 
%$\sigma_i = \left\{a_{i,1} > \ldots > a_{i, n_i}\right\}$. Then we have that 
%\[
%unsd([\sigma_0 \subseteq \ldots \subseteq \sigma_k]) = 
%[a_{0, n_0}, \ldots, a_{k,n_k}].
%\]
%The right hand side of this equality is bounded in $\Theta(L)$, 
%this is a subset of the set of atoms comprising 
%$\sigma_k$. Likewise, if each $\sigma_i$ we taken so that they were in 
%$\Theta(B)$, then we would have the right hand side of 
%the equality to be bounded in $\Theta(B)$. Due to the fact 
%that the horizontal maps are inclusions, it follows that the top square 
%commutes as well. 
\end{proof}

\begin{proposition}
For any $\alpha \in B$, the map 
\[
\widetilde{H}(\psi): \widetilde{H}_{i}(\Delta(B_{<\alpha},k) 
\lra \widetilde{H}_{i}(\Theta(B_{<\alpha}),k)
\] is an isomorphism.
\end{proposition}

\begin{proof}
All maps involved in \ref{eq:A} are homology isomorphims 
with the exception of the map 
$sd(\psi): sd(\Delta(B)) \lra sd(\Theta(B))$. 
By the commutativity of the diagram,  
$sd(\psi)$ is also a homology isomorphism, hence so is $\psi$. %Hence 
%$\Delta(B)$ is homologically isomorphic to $\Theta(B)$. 
\end{proof}

\section{The proof of the main theorem}
\begin{definition}
For $\lambda \in B$, let $G_\lambda$ be the full simplex on the 
$A_\lambda = \left\{a \in A : a \leq \lambda\right\}$. 
Fix $\alpha \in B$. For each $\lambda < \alpha$, define $\Theta_{\alpha, \lambda}$ to be the complex
\[
\Theta_{\alpha, \lambda} = G_{\lambda} \bigcap 
(\bigcup_{\lambda \neq \beta \lessdot_{B} \alpha}G_{\beta})
\]
\end{definition}

While it is not necessarily true that 
$\Theta_{\alpha, \lambda} \subset \Theta(B_{<\lambda})$, 
it is true that 
$\Theta_{\alpha, \lambda} \subset \Theta(L_{<\lambda})$. 
Denote this inclusion map by $j$. Further, we know that the inclusion 
$
\Theta(B_{<\lambda}) \lra 
   \Theta(L_{<\lambda})
$ 
is a homology isomorphism. Let 
$g_*$ be the homology inverse of this map. We set %can define 
\[
\mu_* = g_* \circ j_*: \widetilde{H}_{*}(\Theta_{\alpha, \lambda}, k) \lra 
                      \widetilde{H}_{*-1}(\Theta(B_{<\lambda}), k).
\]
Let 
$
\partial_{i}^{\alpha, \lambda}: 
\widetilde{H}_{i}(\Theta(B_{<\alpha}), k) \lra 
\widetilde{H}_{i-1}(\Theta_{\alpha, \lambda}, k)
$ 
be the connecting homomorphism in the Mayer-Vietoris sequence for the triple 
$
\bigl(G_{\lambda}, \bigcup_{\lambda \neq \beta \lessdot_{\mathscr{B}} \alpha}G_{\beta}, 
            \Theta(B_{<\alpha}) \bigr).
$

%\begin{proposition}
%The following diagram is commutative
%\[
%\begin{CD}
%\widetilde{H}_{i}(\Delta_\alpha, k) @>{\psi_i}>> 
%\widetilde{H}_{i}(\Theta(B_{<\alpha}, k)) \\
%@V{\phi_i}VV                                     
%@V{\psi \circ \partial_{i-2}^{\alpha, \beta}}VV \\
%\widetilde{H}_{i-1}(\Delta_{\beta}, k) @>{\psi_{i-1}}>>  
%\widetilde{H}_{i-1}(\Theta(B_{<\beta}, k)). 
%\end{CD}
%\]
%\end{proposition}

Consider any sequence $\mathcal{S}$ of morphisms of free multigraded modules. 
This can be decomposed as 
\[
\mathcal{S}=\bigoplus_{\alpha \in \mathbb{Z}^n}S_\alpha
\]
where each $S_\alpha$ is a sequence of maps of vector spaces called the multigraded strand of $\mathcal{S}$ in degree 
$\alpha$. Let $(S_\alpha)_i$ denote the $i$th vector space. Notice also that
\[
(\m\mathcal{S})_\alpha = \sum_{\beta \lessdot \alpha}x^{\alpha - \beta}S_\beta \subset S_\alpha.
\]
We will identify $x^{\alpha - \beta}S_\beta$ with $S_\beta$ so one may write $S_\beta \subset S_\alpha$ for $\beta \lessdot \alpha$.

Now we suppose that $I$ is a monomial ideal over the polynomial ring $R=k[x_1, \ldots, x_m]$. Suppose that $\mathcal{F}$ is its minimal free resolution and for each $i$, let $B_i$ be the chosen basis of $F_i$ so that for each $v \in B_i$, we have that 
\[
\partial_i(v)=\sum_{t \in B_{i-1}}[v:t]\cdot t
\]
is such that if $[v:t] \neq 0$ then $mdeg(t) \lessdot_{B} mdeg(v)$.
Let $F_{i, \alpha}$ be the free submodule of $F_i$ spanned by the set 
\[
B_{i, \alpha} = \left\{v \in B_i : mdeg(v)=\alpha\right\}
\] 
Then we have that 
\[
F_i = \bigoplus_{\alpha \in B} F_{i, \alpha}.
\]
Using the notation  
\[
V_{i, \beta} = k\left\langle v: v \in B_{i, \beta} \right\rangle
\]
we write
\[
(F_{\alpha})_i = \bigoplus_{\beta \leq \alpha}V_{i, \beta}
\]
where $x^{\alpha - \beta}V_{i, \beta}$ is identified with $V_{i, \beta}$. In 
particular, $F_i = \bigoplus V_{i, \beta} \otimes R(-\beta)$.

Let $\mathcal{T}$ denote the Taylor resolution \cite{Ta} of $I$. Then there is an exact sequence of the form
\[
0 \lra \sum_{\beta \lessdot \alpha} \mathcal{T}_\beta \lra \mathcal{T}_\alpha 
\lra \mathcal{T}_{\alpha}/\sum_{\beta \lessdot \alpha} \mathcal{T}_\beta \lra 0,
\]
and since the complex $\mathcal{T}_\alpha$ is acylcic, this yields an isomorphism upon passing to homology
\[
H_{i}(\mathcal{T}_{\alpha}/\sum_{\beta \lessdot \alpha} \mathcal{T}_\beta) \overset{\mu_i}\lra
H_{i-1}(\sum_{\beta \lessdot \alpha} \mathcal{T}_\beta) \cong \widetilde{H}_{i-1}(\Theta(L_{<\alpha}),k)
\]
for $i \geq 2$. This isomorphism is defined by 
\[
\mu_i([\overline{v}]) = [\partial^{\mathcal{T}}(v)]
\]
whenever $\overline{v}$ is a cycle in 
$\mathcal{T}_{\alpha}/\sum_{\beta \lessdot \alpha} \mathcal{T}_\beta$ represented by 
the element $v \in \mathcal{T}_\alpha$. Observe that we have 
$H_{0}(\sum_{\beta \lessdot \alpha} \mathcal{T}_{\beta}) = H_{0}(\Theta(L_{\alpha}),k)$ and hence also an isomorphism
\[
H_{1}(\mathcal{T}_{\alpha}/\sum_{\beta \lessdot \alpha}\mathcal{T}_{\beta})   \os{\mu_0}\lra
\widetilde{H}_{0}(\Theta(L_{<\alpha}),k).
\]
Let us make the following identifications in the 
minimal free resolution $\mathcal{F}$ of $I$:
{\small
\[
(F_{\alpha}/(\m F)_{\alpha})_{i} = 
\big(
F_{\alpha}/\sum_{\beta \lessdot_{B} \alpha} F_\beta 
\big)_{i} 
\\
= 
(F_\alpha)_{i}/
\sum_{\beta \lessdot_{B} \alpha}(F_{\beta})_{i} 
\\
= 
\bigoplus_{\beta \leq_{B} \alpha} V_{i, \beta} / 
\bigoplus_{\beta <_{B} \alpha}V_{i, \beta} 
\\
= V_{i , \alpha}.
\]}

Let us fix an embedding of $\mathcal{F}$ into $\mathcal{T}$. 
Then $\mathcal{T} = \mathcal{F} \bigoplus \mathcal{E}$ for some split 
exact complex of multigraded free modules $\mathcal{E}$. 

\begin{proposition}
Suppose $I$ is Betti-linear. Fix $\gamma \lessdot_{B} \alpha$. Then for each $i \geq 1$, 
the diagram:
%\begin{equation}
\[
\begin{CD}
V_{i,\alpha} @>{\partial}>> 
(\sum_{\beta\lessdot_{B}\alpha}\mathcal{F}_{\beta})_{i-1} @>{proj_\gamma}>> 
V_{i-1,\gamma} 
\\
@V{incl}VV     @.      @VV{incl}V  
\\
Z_{i}(\mathcal{T}_{\alpha}/\sum_{\beta \lessdot_{B} \alpha}\mathcal{T}_{\beta})  
@.             @. 
Z_{i-1}(\mathcal{T}_\gamma / \sum_{\nu \lessdot_{B} \gamma}\mathcal{T}_{\nu}) 
\\
@V{proj}VV     @.      @VV{proj}V 
\\
H_{i}(\mathcal{T}_{\alpha}/\sum_{\beta \lessdot_{B} \alpha}\mathcal{T}_{\beta})   
@.             @. 
H_{i-1}(\mathcal{T}_\gamma / \sum_{\nu \lessdot_{B} \gamma}\mathcal{T}_{\nu}) 
\\
@V{\mu_i}VV    @.      @VV{\mu_{i-1}}V 
\\
H_{i-1}(\sum_{\beta \lessdot_{B} \alpha}\mathcal{T}_{\beta})
@.             @. 
H_{i-2}(\sum_{\nu \lessdot_{B} \gamma}\mathcal{T}_{\nu}) 
\\
@|             @.      @|              
\\
\widetilde{H}_{i-1}(\Theta(L_{<\alpha}), k)        
@.             @.      
\widetilde{H}_{i-2}(\Theta(L_{<\gamma}), k)  
\\
@V{\iota^{-1}_{*}}VV           @.      @|     
\\
\widetilde{H}_{i-1}(\Theta(B_{<\alpha}), k) @>>>
\widetilde{H}_{i-2}(\Theta_{\alpha, \gamma}, k) @>>>
\widetilde{H}_{i-2}(\Theta(L_{<\gamma}), k)  
\end{CD}
\]
%\end{equation} 
is commutative.
\end{proposition}

\begin{proof}
Suppose that $v \in V_{i, \alpha}$. 
Note that the assumption of Betti-linearity yields 
\[
\partial^{\mathcal{F}}(v) = \sum_{\beta \lessdot_{B} \alpha}v_{\beta}
\]
with each $v_{\beta} \in V_{i-1, \beta}$. Going down the left hand vertical side of the 
diagram we have $incl(v)=v$, 
and $proj(v)=[v]$. Following this, 
\[
\mu_{i}([v])=[\partial^{\mathcal{T}}(v)]=
[\partial^{\mathcal{F}}(v)]=[\sum_{\beta \lessdot_{B} \alpha}v_{\beta}].
\]
Recall that $\mu_{i}$ is an 
isomorphism for each $i$. The map 
$
\widetilde{H}_{i-2}(\Theta(L_{<\alpha}), k) \lra 
\widetilde{H}_{i-2}(\Theta(B_{<\alpha}), k)
$ 
sends 
$[\sum_{\beta \lessdot_{B} \alpha}v_{\beta}]$ to the class 
$[\sum_{\beta \lessdot_{B} \alpha}v_{\beta} + \partial \tau]$ 
for some boundary $\partial \tau \in \Theta(L_{<\alpha})$. 
Continuing along the bottom row, we write
$\tau = \tau' + \tau''$ where $\tau'$ is a chain in $\Theta(L_{\leq \alpha})$ and $\tau''$ has no faces in $\Theta(L_{\leq \alpha})$.  Then we have that 
$\partial \tau = \partial \tau' + \partial \tau''$, and
\[
\partial_{i-1}^{\alpha, \gamma}
\bigl([\sum_{\beta \lessdot_{B} \alpha}v_{\beta} + \partial \tau]\bigr)= 
\bigl[d_{i-2}(v_{\gamma} + d_{i-1}\tau'')\bigr]
\]
where $d_{i-2}$ is the usual simplicial boundary map.
We see that 
\[
\bigl[d_{i-2}(v_{\gamma} + d_{i-1}(\tau'')\bigr] = 
\bigl[d_{i-2}(v_\gamma)\bigr] \in \widetilde{H}_{i-3}(\Theta_{\alpha, \gamma}, k),
\]
which is mapped to the same class considered in 
$\widetilde{H}_{i-3}(\Theta(L_{<\gamma}), k)$.

Chasing the diagram in the other direction, we have that 
\[
\partial^{\mathcal{F}}(v) = \sum_{\beta \lessdot_{B} \alpha}v_{\beta}
\]
hence
\[
proj_{\gamma}(\sum_{\beta \lessdot_{B} \alpha}v_\beta )= v_{\gamma}
\]
Of course, we know that $incl(v_{\gamma})=v_{\gamma}$, and hence upon 
passing to homology we have $[v_{\gamma}]$. 
We then have that $\mu_{i-1}([v_{\gamma}])=[\partial^{\mathcal{T}}(v_\gamma)]$.

Finally, we must observe that 
$[\partial^{\mathcal{T}}(v_\gamma)] = [d_{i-2}(v_\gamma)]$ in 
$\widetilde{H}_{i-3}(\Theta(L_{< \gamma}), k)$.
\end{proof}

%\begin{Main Theorem}
%Let $I$ be an ideal of $R$ generated by monomials.
%$I$ is Betti-linear if and only if $\mathscr{F}$ can be recovered from the poset 
%construction applied to the Betti poset of the ideal $I$.
%\end{Main Theorem}
\begin{proof}[Proof of Theorem~\ref{main theorem}:]
We will show that the following diagram is commutative for each $i \geq 1$:
\small{
\[
\begin{CD}
\cdots \rightarrow \bigoplus \widetilde{H}_{i-2}(\Delta(B_{< \lambda}),k) \otimes R(-\lambda)
       @>>> \bigoplus \widetilde{H}_{i-3}(\Delta(B_{< \lambda}),k) \otimes R(-\lambda)
       \rightarrow \cdots 
\\
    @V{\sum{\psi_i} \otimes 1}VV      @VV{\sum{\psi_{i-1}} \otimes 1}V        
\\
\cdots \rightarrow \bigoplus \widetilde{H}_{i-2}(\Theta(B_{< \lambda}),k) \otimes R(-\lambda)
       @>>> \bigoplus \widetilde{H}_{i-3}(\Theta(B_{< \lambda}),k) \otimes R(-\lambda)
       \rightarrow \cdots 
\\ 
    @V{\sum{f_{i,\lambda} \otimes 1}}VV      @VV{\sum{f_{i-1,\lambda}} \otimes 1}V        
\\
\cdots \rightarrow \bigoplus V_{i, \lambda} \otimes R(-\lambda)                                         
       @>>> \bigoplus V_{i-1, \lambda} \otimes R(-\lambda)                         
       \rightarrow \cdots 
\end{CD}
\]
}
However, the top square of this diagram is commutative due to Proposition 4.6.
The map $f_{i, \lambda}$ is the following composition:
\[
f_{i, \lambda}: \widetilde{H}_{i-2}(\Theta(B_{<\lambda}), k) \lra \widetilde{H}_{i-2}(\Theta(L_{<\lambda}), k) \lra H_{i}(\mathcal{T}_{\lambda}/\sum_{\beta \lessdot_{B} \lambda}\mathcal{T}_{\beta}) \lra V_{i, \lambda}
\]
where all arrows involved are the inverses of the vertical maps from Proposition 5.2. Now, 
Proposition 5.2 immediately implies the commutativity of the bottom square. Since each $f_{i,\lambda}$ and each $\psi_i$ are isomorphisms, the result follows.
\end{proof}

\end{document}